\documentclass[11pt]{amsart}

\usepackage[letterpaper,top=3cm,bottom=3cm,left=3cm,right=3cm,marginparwidth=1.75cm]{geometry}
\usepackage{amsmath,amssymb,amsthm,amsfonts}
\usepackage{graphicx}
\usepackage[colorlinks=true, allcolors=blue]{hyperref}
\usepackage{xcolor}
\usepackage{enumerate}
\usepackage{mathrsfs}
\usepackage{tikz}
\usetikzlibrary{trees,calc}

\title[Inverse Fault Friction Problem]{Stable Recovery of Coefficients in an Inverse Fault Friction Problem}

\author[M. V. de Hoop]{Maarten V. de Hoop}
\address{Maarten V. de Hoop: Computational and Applied Mathematics and Earth Science, Rice University, Houston, TX 77005, USA}
\email{mdehoop@rice.edu}

\author[M. Lassas]{Matti Lassas}
\address{Matti Lassas: Department of Mathematics and Statistics, University of Helsinki, FI-00014 Helsinki, Finland}
\email{matti.lassas@helsinki.fi}

\author[J. Lu]{Jinpeng Lu}
\address{Jinpeng Lu: Department of Mathematics and Statistics, University of Helsinki, FI-00014 Helsinki, Finland} 
\email{jinpeng.lu@helsinki.fi}

\author[L. Oksanen]{Lauri Oksanen}
\address{Lauri Oksanen: Department of Mathematics and Statistics, University of Helsinki, FI-00014 Helsinki, Finland} 
\email{lauri.oksanen@helsinki.fi}

\theoremstyle{definition}
\newtheorem{definition}{Definition}[section]

\theoremstyle{plain}
\newtheorem{theorem}{Theorem} 
\newtheorem{lemma}[definition]{Lemma} 
 
\newtheorem{proposition}[definition]{Proposition}

\newcommand\p{\partial}

\newcommand{\veu}{\boldsymbol{u}}
\newcommand{\vev}{\boldsymbol{v}}
\newcommand{\vew}{\boldsymbol{w}}
\newcommand{\verho}{\boldsymbol{\rho}}
\newcommand{\R}{{\mathbb R}}

\newcommand{\Surf}{{\Sigma}}

\newcommand{\frc}{{\mathrm{f}}}

\newcommand{\mathbi}[1]{{\boldsymbol{#1}}}

\newcommand{\jmp}[1]{{\,[\,#1\,]_{-}^{+} }}

\newcommand{\ven}{\mathbi{n}}

\begin{document}

\maketitle

\vspace{-5mm}
\begin{abstract}
We consider the inverse fault friction problem of determining the friction coefficient in the Tresca friction model, which can be formulated as an inverse problem for differential inequalities. We show that the measurements of elastic waves during a rupture uniquely determine the friction coefficient at the rupture surface with explicit stability estimates.
\end{abstract}

\section{Introduction}

The study of earthquake physics remains highly challenging through its complex dynamics and multifaceted nature. Nearly all aspects of earthquake ruptures are controlled by the friction along a fault, where these commonly occur, that progressively increases with tectonic forcing. Indeed, in a recent Annual Review of Earth and Planetary Sciences, it was stated that ``determining the friction during an earthquake is required to understand when and where earthquakes occur'' (Brodsky \textit{et al}. \cite{B20}). Some common approach has been developed retrieving the stress evolution at each point of the fault as dictated by the slip history obtained from the kinematic inverse rupture problem; we mention work by Ide and Takeo \cite{IT}, who determined the spatiotemporal slip distribution on an assumed fault plane of the 1995 Kobe earthquake by ``waveform inversion'' and then numerically solved the elastodynamic equations to determine the stress distribution and constitutive relations on the fault plane. 
However, seismologists studying earthquake dynamics have reported that both stress and friction on a fault are still poorly known and difficult to constrain with observations (Causse, Dalguer and Mai \cite{CDM}). Here, we address the question whether this is possible, in principle. 

We study the recovery of a time- and space-dependent friction coefficient via the slip rate and normal and tangential stresses, using the Tresca model (see e.g. the book of Sofonea and Matei \cite{SM}), at a pre-existing fault from ``near-surface'' elastic-wave, that is, seismic displacement data. 
This dynamic inverse friction problem can be regarded as an inverse problem for differential inequalities, as the Tresca friction model can be formulated through variational inequalities as seen in many contact mechanics problems (e.g. \cite{DL,SM}). While inverse problems for differential equations have been widely studied, inverse problems for differential inequalities have not yet received much attention.
Our approach is based on the quantitative unique continuation for the elastic wave equation established in our recent work \cite{DLLO}, where we studied the kinematic inverse rupture problem of determining the friction force at the rupture surface from seismic displacement data.
Itou and Kashiwabara \cite{IK} recently analyzed the Tresca model on a fault coupled to the elastic wave equation; we exploit their results in our study of the inverse problem.
We also mention recent work by Hirano and Itou \cite{HI} on deriving an analytical solution to the slip rate distribution of self-similar rupture growth under a distance-weakening friction model.
As a disclaimer, while we address the most fundamental question, we do ignore more complex physics such as thermo-mechanical effects.

We remark that in the past two decades, quite many studies have been devoted to the practical determination of fault frictional properties by analyzing slowly-evolving afterslip following large earthquakes. For very recent results, see Zhao and Yue \cite{ZY}. Afterslip is the fault slip process in response to an instantaneous coseismic stress change, in which the slip velocity decrease corresponds to the stress releasing by itself. Its ``self-driven'' nature provides a framework to model the slip process with the fault frictional properties alone. For a review, see Yue et al. \cite{YZ}. Afterslip is analyzed with quasi-static deformation, that is, with the elastostatic system of equations, typically using geodetic data.

\medskip
Let $M\subset \mathbb{R}^3$ be a bounded connected open set with smooth boundary, modelling the solid Earth. Let $\overline{\Sigma_{{\rm f}}}\subset M$ be a connected orientable embedded smooth surface with nonempty smooth boundary satisfying $\overline{\Sigma_{{\rm f}}}\cap \partial M=\emptyset$, modelling the rupture surface.
Consider the elastic wave equation
\begin{equation} \label{eq-elastic-special}
\rho \partial_t^2 \veu -\mu \Delta \veu -(\lambda+\mu)\nabla (\nabla\cdot \veu)=0 \quad \textrm{in } \big(M \setminus \overline{\Surf_\frc}\big) \times (-T,T) \, ,
\end{equation}
with the Tresca friction condition (e.g. \cite{DL,IK,KI22}) on the rupture surface $\Sigma_{{\rm f}}$:
\begin{equation} \label{Tresca}
\left\{ \begin{aligned}
&\sigma_n=F_n \, \textrm{ is given}, \\
&\jmp{\sigma_n}=0,\quad \jmp{\boldsymbol{\sigma}_{\tau}}=0, \\
&|\boldsymbol{\sigma}_{\tau}|<g \Longrightarrow \jmp{\p_t \veu_{\tau}}=0,\\
&|\boldsymbol{\sigma}_{\tau}|=g \Longrightarrow \jmp{\p_t \veu_{\tau}}\cdot \boldsymbol{\sigma}_{\tau}=g\,\big| \jmp{\p_t \veu_{\tau}} \big| \, .
\end{aligned} \right.
\end{equation}
Here $\sigma_n=(\boldsymbol{\sigma}(\veu) \ven)\cdot \ven$ and $\boldsymbol{\sigma}_{\tau}=\boldsymbol{\sigma}(\veu) \ven-\sigma_n \ven$ are the normal and tangential components of the stress tensor $\boldsymbol{\sigma}(\veu)$, where $\ven$ is the unit normal vector of the rupture surface $\Sigma_{{\rm f}}$.
The stress tensor $\boldsymbol{\sigma}(\veu)$ is defined as
\begin{equation} \label{def-stress}
\boldsymbol{\sigma}(\veu)=2\mu \boldsymbol{\varepsilon}(\veu)+\lambda \big({\rm tr} \,\boldsymbol{\varepsilon}(\veu) \big) \boldsymbol{I}, \quad\;
\boldsymbol{\varepsilon}(\veu)=\frac12 \Big(\nabla \veu+(\nabla \veu)^T \Big).
\end{equation}
The notation $\jmp{\cdot}$ stands for the jump across the rupture surface, more precisely,
\begin{align} \label{tangential-jump}
\jmp{\p_t \veu_{\tau}} := \p_t (\veu_{\tau}^{+} - \veu_{\tau}^{-}),
\quad 
\veu_{\tau}^{\pm}:= \lim_{h\to 0^{\pm}} \veu_{\tau} (z+h\ven,t),\quad z\in \Sigma_{{\rm f}},
\end{align}
where $\veu_{\tau}$ is the tangential component of $\veu$, with respect to the unit normal vector field $\ven$ smoothly extended to a neighborhood of $\Sigma_{{\rm f}}$.
The friction force $g$ at $\Sigma_{{\rm f}}$ is of the form
\begin{equation}\label{friction-coeff}
g=\mathscr{F}|F_n|,
\end{equation}
where $\mathscr{F}>0$ is the friction coefficient.
Note that the friction force $g$ and the friction coefficient $\mathscr{F}$ may depend on time.

Regarding the direct problem for the Tresca friction model above, the weak formulation is understood in the variational sense (see e.g. \cite{IK,DL,SM}).
Let $V_0=\{\vev\in H^1(M \setminus \overline{\Surf_\frc}): \vev=0 \textrm{ on }\partial M\}$.
Recall that with the Dirichlet condition on the boundary $\partial M$, 
the problem of finding $\veu$ satisfying the Tresca friction model (\ref{eq-elastic-special}-\ref{Tresca}) is formulated as finding
$\veu$ such that for all $\vev\in H^1(V_0)$ and all $t\in (0,T)$, the following variational inequality holds:
\begin{eqnarray} \label{eq-variational}
&&\rho \int_{M\setminus \overline{\Surf_\frc}} (\vev-\partial_t \veu)\partial_t^2 \veu \, dV +a\big(\veu,\vev-\partial_t \veu \big) \\
&&+\int_{\Surf_\frc} g \cdot \big|\jmp{\vev_{\tau}}\big| \, dA -\int_{\Surf_\frc} g \cdot \big|\jmp{\partial_t \veu_{\tau}}\big| \,dA
+\int_{\Surf_\frc} F_n \cdot \jmp{v_n} \, dA \geq 0, \nonumber
\end{eqnarray}
where $dV,\,dA$ stand for the volume and area element of $M$ and $\Surf_\frc$, respectively, and $a$ is a bilinear symmetric form defined by
$a(\vev,\vew):=\int_{M\setminus \overline{\Surf_\frc}} \textrm{tr}\big(\boldsymbol{\sigma}(\vev)^T 
\boldsymbol{\varepsilon}(\vew)\big) \, dV$.
If one assumes
$$\veu_0:=\veu(\cdot,0)\in H^1(	V_0), \quad \partial_t\veu (\cdot,0)\in H^1(V_0),$$
$$F_n\in H^2(0,T; L^2(\Surf_\frc)), \quad g\in  H^2(0,T; L^2(\Surf_\frc)),$$
and the compatibility conditions at $t=0$:
\begin{equation*}
\left\{ \begin{aligned} &\textrm{div}\, \boldsymbol{\sigma} (\veu_0) \in L^2(M \setminus \overline{\Surf_\frc}), \\
&\sigma_n(\veu_0^{+})=\sigma_n(\veu_0^{-})=F_n(\cdot,0) \quad &\textrm{ on } \Surf_\frc, \\
&\boldsymbol{\sigma}_{\tau}(\veu_0^{+})=\boldsymbol{\sigma}_{\tau}(\veu_0^{-})=0,\;\; \partial_t \veu_{\tau}(\cdot,0)=0 \quad &\textrm{ on } \Surf_\frc ,
\end{aligned} \right.
\end{equation*}
then the friction problem above has a unique solution \cite{IK},
\begin{equation} \label{regularity-direct}
\veu\in W^{1,\infty} \big(0,T;H^1(V_0) \big)\cap W^{2,\infty}\big(0,T; L^2(M \setminus \overline{\Surf_\frc}) \big).
\end{equation}

\smallskip
In this paper, we consider the inverse problem of determining the friction coefficient $\mathscr{F}$ in the Tresca friction model above.
For the sake of presentation, we consider the elastic wave equation \eqref{eq-elastic-special} on the time interval $(-T,T)$.
Let $M\subset \mathbb{R}^3$ be a bounded connected open set with smooth boundary, modelling the solid Earth. Let $\overline{\Sigma_{{\rm f}}}\subset M$ be a connected orientable embedded smooth surface with nonempty smooth boundary, modelling the rupture surface.
Denote the \emph{a priori} bound for the norm of the elastic wave $\veu$ in the space \eqref{regularity-direct} by
\begin{equation} \label{u-norm}
\|\veu\|_{W^{1,\infty}(-T,T;H^1(M \setminus \overline{\Surf_\frc}))} +\|\veu\|_{W^{2,\infty}(-T,T;L^2(M \setminus \overline{\Surf_\frc}))} \leq \Lambda_0.
\end{equation}
We impose the following assumption on the regularity of the normal stress $F_n$ and the friction coefficient $\mathscr{F}$.
Assume that $F_n,\mathcal{F}\in C^{0,1}(\Sigma_{{\rm f}}\times [-T,T])$, and
\begin{equation} \label{apriori-fault}
|F_n| \geq c_0>0,\quad \|F_n\|_{C^{0,1}(\Sigma_{{\rm f}}\times [-T,T])}\leq C_0, \quad \|\mathscr{F}\|_{C^{0,1}(\Sigma_{{\rm f}}\times [-T,T])}\leq C_0.
\end{equation}
In addition, assume that the parameters $\rho,\lambda,\mu$ in the elastic wave equation \eqref{eq-elastic-special} to be smooth and time-independent on $M \setminus \overline{\Surf_\frc}$.

\smallskip
We prove the following result on the inverse fault friction problem during a rupture.

\begin{theorem}\label{main}
Let $M\subset \mathbb{R}^3$ be a bounded connected open set with smooth boundary. Let $\overline{\Sigma_{{\rm f}}}\subset M$ be a connected orientable embedded smooth surface with nonempty smooth boundary, satisfying $\overline{\Sigma_{{\rm f}}}\cap \partial M=\emptyset$.
Consider the elastic wave equation \eqref{eq-elastic-special} with smooth time-independent parameters and the Tresca friction condition \eqref{Tresca}.
Assume the normal stress $F_n$ and the friction coefficient $\mathscr{F}$ satisfy \eqref{apriori-fault} on $\Sigma_{{\rm f}}$. 
Assume $\big| \jmp{\p_t \veu_{\tau}} \big| \geq c_1$ on $\Sigma_{{\rm f}}\times [-T,T]$ for some constant $c_1>0$ and a priori norm \eqref{u-norm} for the elastic wave $\veu$.
Suppose that we are given the elastic wave $\veu$ on an interior open set $U$ satisfying $\overline{U}\subset M\setminus \overline{\Sigma_{{\rm f}}}$ up to sufficiently large time $T$.
Then we have the following conclusions.

\begin{enumerate}[(1)]

\item The friction coefficient $\mathscr{F}$ on $\Sigma_{{\rm f}}\times [-\frac{T}{2},\frac{T}{2}]$ is uniquely determined by $\veu$ on $U\times [-T,T]$.

\smallskip
\item Suppose that we have two systems with friction coefficients $\mathscr{F}_1,\mathscr{F}_2$, and we are given the corresponding elastic waves $\veu_1,\veu_2$ on $U\times [-T,T]$.
Then there exist constants $\widehat{\varepsilon_0},C,c>0$ such that for any $0<\varepsilon_0<\widehat{\varepsilon_0}$,
if
$$\|\veu_1-\veu_2\|_{H^2(U\times [-T,T])}\leq \varepsilon_0,$$
then the friction coefficients satisfy
$$\|\mathscr{F}_1-\mathscr{F}_2\|_{L^2(\Sigma_{{\rm f}}\times [-\frac{T}{2},\frac{T}{2}])} \leq C(\log|\log\varepsilon_0|)^{-c},$$
where the constants $\widehat{\varepsilon_0}, C$ depend on $c_0,C_0,\Lambda_0,c_1,T$, parameters of the elastic wave equation and $M,U,\Sigma_{{\rm f}}$, and $c$ is an absolute constant.
\end{enumerate}
\end{theorem}

We remark that one could also formulate the result assuming a lower bound for $\big| \jmp{\p_t \veu_{\tau}} \big|$ on a subset of $\Sigma_{{\rm f}}$ and then recovers the friction coefficient $\mathscr{F}$ in the same subset.
Theoretically speaking, our method can also work with measurements of the elastic waves on an open subset of the boundary $\partial M$ if additional regularity of the elastic waves on $\partial M$ can be assumed, see \cite[Remark 1]{DLLO} or \cite[Theorem 3]{BILL}; however, this additional boundary regularity required is not provided by the regularity class \eqref{regularity-direct}.

\smallskip
The proof of Theorem \ref{main} is divided into two parts: the measurements of $\veu$ on $U$ determine $\veu$ near $\Sigma_{{\rm f}}$, and the latter determines $\mathscr{F}$ under the Tresca friction condition.
The first part, also known as the kinematic inverse rupture problem \cite{FSG}, has mostly been done in our recent work \cite{DLLO}. 
Our method was based on the quantitative unique continuation for the elasticity system, motivated by \cite{Tataru1,EINT}.
However, regularity issues remain: the actual regularity of waves is not enough for the quantitative unique continuation arguments to work on the whole domain, which is addressed in Section \ref{sec-regularity}. The second part is discussed in Section \ref{sec-proof}.

\medskip
\noindent {\bf Acknowledgement.} The authors thank Erik Burman for pointing out the reference on the regularity of the direct problem, and thank the anonymous referee for pointing out a mistake in an earlier version of the paper.
MVdH was supported by the Simons Foundation under the MATH + X program, the National Science Foundation under grant DMS-2108175, and the corporate members of the Geo-Mathematical Imaging Group at Rice University.
M.L. and J.L. were supported by the PDE-Inverse project of the European Research Council of the European Union and the Research Council of Finland, grants 273979 and 284715. 
L.O. was supported by the European Research Council of the European Union, grant 101086697 (LoCal), and the Research Council of Finland, grants 347715 and 353096. Views and opinions expressed are those of the authors only and do not necessarily reflect those of the European Union or the other funding organizations.

\section{Interior regularity} \label{sec-regularity}

Let $N \subset \R^3$ be a bounded connected open set with smooth boundary. Let $\delta>0$ be small, and denote the interior by
\begin{equation} \label{def-Mdelta}
N_{\delta}:=\{x\in N: d(x,\partial N)\geq \delta\}.
\end{equation}
Let $\veu$ be an elastic wave satisfying the elastic wave equation \eqref{eq-elastic-special} on the whole set $N$ (without the presence of a rupture surface).
Recall \cite[Lemma 5.1]{EINT} that the elastic wave equation can be decomposed into a system of hyperbolic equations for $(\veu,\textrm{div}\, \veu, \textrm{curl}\, \veu)$.
Let $U\subset N$ be a connected open subset, and we choose $\delta$ sufficiently small such that $\overline{U}\subset N_{\delta}$.

\begin{figure}[h]
  \begin{center}
    \includegraphics[width=0.42\linewidth]{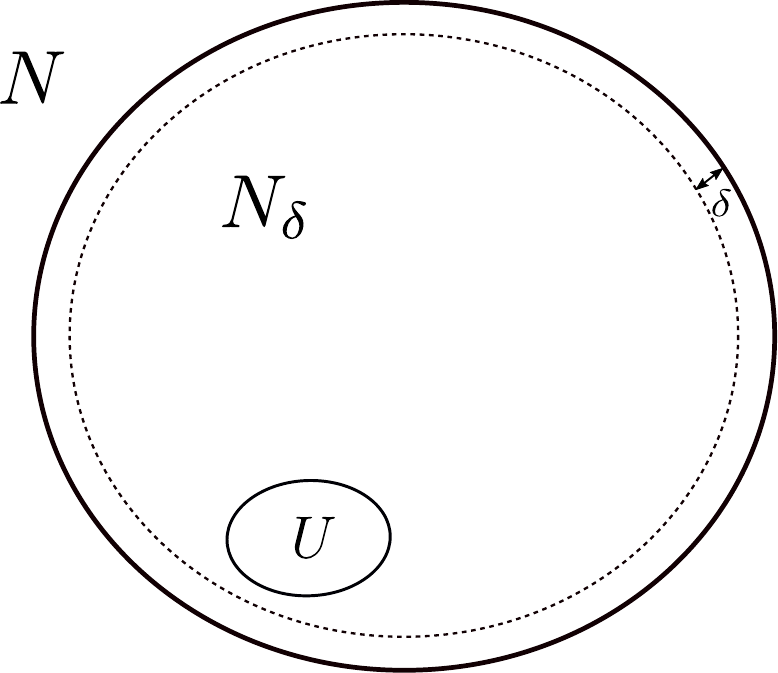}
    \caption{Setting of Lemma \ref{lemma-interior-regularity} and Proposition \ref{uc-M}.}
    \label{fig_Mdelta}
  \end{center}
\end{figure}

We apply Proposition 3.2 in \cite{DLLO} to $(\veu,\textrm{div}\, \veu, \textrm{curl}\, \veu)$ in $N_{\delta}$: for sufficiently small $h$ and sufficiently large $T$ (specified in \cite[Proposition 3.2]{DLLO}), if 
$$\|\veu\|_{H^2(U\times [-T,T])} \leq \varepsilon_0,$$
then
\begin{equation} \label{uc-delta}
\|\veu\|_{L^2(N_{\delta}\times [-\frac{T}{2},\frac{T}{2}])} \leq C\exp(h^{-cn}) \frac{\Lambda_{\delta}}{\Big( \log\big(1+h\frac{\Lambda_{\delta}}{\varepsilon_0}\big)\Big)^{\frac12}}+C\|\veu\|_{H^1(N_{\delta}\times [-T,T])}\, h^{\frac{1}{n+1}},
\end{equation}
where $n=\textrm{dim}(N)=3$, and
\begin{equation}
\Lambda_{\delta}:=\big\|(\veu,\textrm{div}\, \veu, \textrm{curl}\, \veu) \big\|_{H^1(N_{\delta}\times [-T,T])}.
\end{equation}

Observe that $\Lambda_{\delta}$ essentially asks for $H^2$ regularity, while the solution of the original direct problem \eqref{regularity-direct} is only in $H^1$ in space. One way to resolve the issue is to use interior regularity estimate, and then use Sobolev embedding to get an immediate estimate for the boundary layer.

We assume that the elastic wave $\veu$ in $N$ is in the following energy class,
\begin{equation} \label{u-norm-N}
\|\veu\|_{W^{1,\infty}(-T,T;H^1(N))} +\|\veu\|_{W^{2,\infty}(-T,T;L^2(N))} \leq \Lambda_0.
\end{equation}

\begin{lemma}\label{lemma-interior-regularity}
The following $H^2$ interior regularity estimate holds for the elastic wave equation \eqref{eq-elastic-special} on $N\times [-T,T]$:
$$\|\veu\|_{H^2(N_{2\delta}\times [-T,T])} 
\leq C_{T} C_{M} \delta^{-4} \Big( \|\veu\|_{W^{1,\infty}(-T,T;H^1(N))} +\|\veu\|_{W^{2,\infty}(-T,T;L^2(N))} \Big),$$
which gives a bound
\begin{equation}\label{norm-H2}
\Lambda_{\delta}\leq C_T C_N \delta^{-4}\Lambda_0,
\end{equation}
where $C_T$ depends only on $T$, and $C_N$ depends only on geometric parameters of $N$.
\end{lemma}

\begin{proof}
Suppose that $\varphi$ is a weak solution of the elliptic equation on $N$,
$$\Delta \varphi=f,\quad \varphi\in H^1(N),\; f\in L^2(N).$$
Let $\chi$ be a cut-off function satisfying $\chi|_{N_{\delta}}=1$, $\chi|_{\partial N}=0$ and $\|\chi\|_{C^2(N)}\leq C \delta^{-2}$.
Then $$\varphi_0:=\chi \varphi$$ satisfies
\begin{equation*}
\left\{ \begin{aligned} &\Delta \varphi_0 =\chi \Delta \varphi +[\Delta,\chi]\varphi=:\widetilde{f}, \\
&\varphi_0|_{\partial N}=0,
\end{aligned} \right.
\end{equation*}
where 
$$\|\widetilde{f}\|_{L^2(N)}\leq C\delta^{-2} \Big(\|\varphi\|_{H^1(N)}+\|f\|_{L^2(N)} \Big).$$
Hence $\varphi_0\in H^2(N)$ by boundary regularity for elliptic equations (e.g. \cite[Theorem 4 in Chapter 6.3]{Evans}). Then it follows that
$\varphi|_{N_{\delta}}\in H^2(N_{\delta})$, and
\begin{equation} \label{interior-regularity}
\|\varphi\|_{H^2(N_{\delta})}\leq\|\varphi_0\|_{H^2(N)}\leq C_N \delta^{-2}\Big(\|\varphi\|_{H^1(N)}+\|f\|_{L^2(N)} \Big).
\end{equation}
The constant $C_N$ depends on the first Dirichlet eigenvalue on $N$ which is uniformly bounded below by diameter and curvature bounds (e.g. \cite[Theorem 8]{LY}).
The same argument is valid for $\varphi\in L^2(N)$ and $f\in H^{-1}(N)$ with constant $C_N\delta^{-2}$.

\smallskip
Now we switch to the notations in our first paper \cite{DLLO},
\begin{equation} \label{def-uvw}
(\veu,v,\vew):=(\veu, \textrm{div}\, \veu, \textrm{curl}\, \veu).
\end{equation}
Recall that the elastic wave equation \eqref{eq-elastic-special} can be decomposed into the following system (\cite[Lemma 5.1]{EINT} or \cite[Lemma A.1]{DLLO}):
\begin{equation} \label{eq-system}
\left\{ \begin{aligned}
&\frac{\rho}{\mu}\partial_t^2 \veu -\Delta \veu+A_1(\veu, v)=0, \\
&\frac{\rho}{2\mu+\lambda}\partial_t^2 v -\Delta v+A_2(\veu, v, \vew)=0, \\
&\frac{\rho}{\mu}\partial_t^2 \vew -\Delta \vew+A_3(\veu, v, \vew)=0,
\end{aligned} \right.
\end{equation}
where $A_i$ are first order and has no time derivative. 

Then consider the second equation. Since $\veu\in W^{1,\infty}(-T,T;H^1(N))$ from \eqref{u-norm-N}, then
$$v,\vew\in W^{1,\infty}(-T,T;L^2(N)), \quad A_2(\veu,v,\vew)\in W^{1,\infty}(-T,T;H^{-1}(N)).$$ 
On the other hand, since $\veu\in W^{2,\infty}(-T,T; L^2(N))$ from \eqref{u-norm-N}, then
$$v\in W^{2,\infty}(-T,T;H^{-1}(N)), \quad \partial_t^2 v\in L^{\infty}(-T,T;H^{-1}(N)).$$ 
Then using the second equation gives 
$$\Delta v\in L^{\infty}(-T,T;H^{-1}(N)).$$
Thus from the interior regularity argument \eqref{interior-regularity} (for $L^2$-$H^{-1}$), we have $$v|_{N_{\delta}}\in L^{\infty}(-T,T;H^1(N_{\delta}))$$ with constant $C_N \delta^{-2}$ in the regularity estimate.

Next consider the first equation on $N_{\delta}$. Using the improved interior regularity for $v|_{N_{\delta}}$, we see that $A_1(\veu,v)|_{N_{\delta}}\in L^{\infty}(-T,T;L^2(N_{\delta}))$. This gives 
$$\Delta \veu|_{N_{\delta}}\in L^{\infty}(-T,T;L^2(N_{\delta})),$$ considering that $\partial_t^2 \veu \in L^{\infty}(-T,T;L^2(N))$.
Thus by shrinking the domain by another $\delta$ and using \eqref{interior-regularity}, we obtain $$\veu|_{N_{2\delta}}\in L^{\infty}(-T,T;H^2(N_{2\delta}))$$ with constant $C_N \delta^{-4}$ in the regularity estimate, i.e.,
\begin{equation} \label{interior-space}
\|\veu\|_{L^{\infty}(-T,T;H^2(N_{2\delta}))} \leq C_N \delta^{-4} \Big( \|\veu\|_{W^{1,\infty}(-T,T;H^1(N))} +\|\veu\|_{W^{2,\infty}(-T,T;L^2(N))} \Big).
\end{equation}
Thus from \eqref{interior-space} and \eqref{u-norm-N}, we obtain the estimate
\begin{eqnarray*}
\|\veu\|_{H^2(N_{2\delta}\times [-T,T])} &\leq& C_T \Big( \|\veu\|_{L^{\infty}(-T,T;H^2(N_{2\delta}))} + \|\veu\|_{W^{1,\infty}(-T,T;H^1(N))} +\|\veu\|_{W^{2,\infty}(-T,T;L^2(N))} \Big) \\
&\leq& C_{T} C_{N} \delta^{-4} \Big( \|\veu\|_{W^{1,\infty}(-T,T;H^1(N))} +\|\veu\|_{W^{2,\infty}(-T,T;L^2(N))} \Big).
\end{eqnarray*}
\end{proof}

\begin{proposition} \label{uc-M}
Let $N\subset \R^3$ be a bounded connected open set with smooth boundary and $U\subset N$ be a connected open set. Let $\veu$ be a solution of the elastic wave equation \eqref{eq-elastic-special} in $N\times [-T,T]$ with a priori norm \eqref{u-norm-N}. 
Then there exist constants $\widehat{\varepsilon_0},C,c>0$ such that for any $0<\varepsilon_0<\widehat{\varepsilon_0}$, if
$$\|\veu\|_{H^2(U\times [-T,T])}\leq \varepsilon_0,$$
then
$$\|\veu\|_{L^2(N\times [-\frac{T}{2},\frac{T}{2}])}\leq C(\log|\log \varepsilon_0|)^{-c},$$
where the constants $\widehat{\varepsilon_0},C$ depend on $\Lambda_0,T$, parameters of the elastic wave equation and geometric parameters of $N$, and $c$ is an absolute constant. 
\end{proposition}
\begin{proof}
We have already estimated the $L^2$-norm of $\veu$ on $N_{\delta}\times [-T/2,T/2]$ from \eqref{uc-delta} and Lemma \ref{lemma-interior-regularity}.
The $L^2$-norm on the remaining part follows from the Sobolev embedding theorem.
Namely, consider
$$\mathcal{N}_{\delta}^T :=\{x\in N: d(x,\partial N)\leq \delta\} \times [-\frac{T}{2},\frac{T}{2}].$$
Apply the Sobolev embedding theorem (e.g. Theorem 4.12 in \cite{AF}) to the space $N\times [-T,T]$ of dimension $n+1=4$, which satisfies the uniform cone condition (Definition 4.8 in \cite{AF}),
$$\|\veu\|_{L^{\frac{2(n+1)}{n-1}}(N\times [-T,T])} \leq C\|\veu\|_{H^1(N\times [-T,T])}\leq C\Lambda_0.$$
Then,
$$\|\veu\|_{L^2(\mathcal{N}_{\delta}^T)} \leq \|\veu\|_{L^{\frac{2(n+1)}{n-1}}(N\times [-T,T])} \big(\textrm{Vol}(\mathcal{N}_{\delta}^T) \big)^{\frac{1}{n+1}} \leq C \Lambda_0 \delta^{\frac{1}{n+1}}.$$
Hence, we have
\begin{eqnarray} \label{L2-N}
\|\veu\|_{L^2(N\times [-\frac{T}{2},\frac{T}{2}])} 
&\leq& \|\veu\|_{L^2(N_{\delta}\times [-\frac{T}{2},\frac{T}{2}])}+ \|\veu\|_{L^2((N\setminus N_{\delta})\times [-\frac{T}{2},\frac{T}{2}])} \nonumber \\
&\leq& C\exp(h^{-cn}) \frac{\delta^{-4}\Lambda_0}{\Big( \log\big(1+h\frac{\delta^{-4}\Lambda_0}{\varepsilon_0}\big)\Big)^{\frac12}}+C\Lambda_{0} h^{\frac{1}{n+1}}+C\Lambda_{0} \delta^{\frac{1}{n+1}}.
\end{eqnarray}

From here, we need to balance the parameters $\delta,h,\varepsilon_0$.
Choose $\delta=h$ such that the three terms on the right-hand side of \eqref{L2-N} are equal, and we get
\begin{equation}
\delta=h=C \big(\log |\log \varepsilon_0| \big)^{-c},
\end{equation}
for some constant $c$ depending only on $n=3$, and for some constant $C$ independent of $h$.
The condition $h<h_0$ gives the choice for $\widehat{\varepsilon_0}$:
\begin{equation}
\widehat{\varepsilon_0}=\Big( \exp \exp \big(C^{-1}h_0^{-1/c} \big) \Big)^{-1}.
\end{equation}
Inserting the choice of $\delta,h$ back into \eqref{L2-N} gives the desired form of the estimate.
\end{proof}

\section{Inverse friction problem} \label{sec-proof}

Using Proposition \ref{uc-M}, we consider the inverse problem of determining the friction coefficient $\mathscr{F}$ in the Tresca friction model.
As Proposition \ref{uc-M} is formulated without the presence of the rupture surface $\Sigma_{{\rm f}}$, we apply the result to a smooth manifold whose boundary extends $\Sigma_{{\rm f}}$. Of course the choice of such a manifold can be arbitrary and in this paper we use a type as illustrated in Figure \ref{fig_N}.

\begin{figure}[h]
  \begin{center}
    \includegraphics[width=0.42\linewidth]{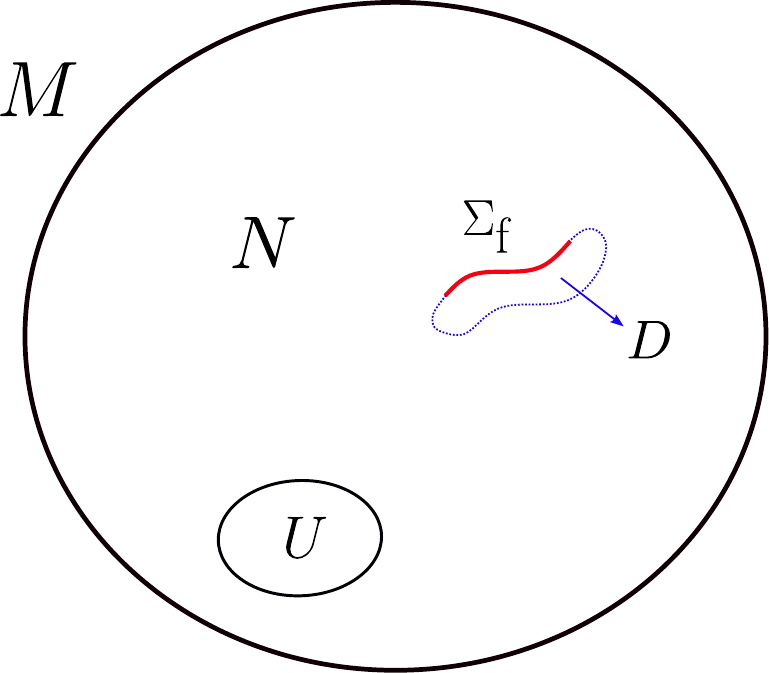}
    \caption{Applying the quantitative unique continuation to the connected open set $N=M\setminus D$ with smooth boundary, where $D$ is a compact smooth manifold whose boundary extends $\Sigma_{{\rm f}}$.}
    \label{fig_N}
  \end{center}
\end{figure}

\begin{lemma} \label{lemma-trace}
Let $M\subset \mathbb{R}^3$ be a bounded connected open set with smooth boundary, $\overline{\Sigma_{{\rm f}}}$ be a connected orientable smooth embedded compact rupture surface with nonempty smooth boundary satisfying $\overline{\Sigma_{{\rm f}}}\cap \partial M=\emptyset$, and $U\subset M$ be a connected open set satisfying $\overline{U}\subset M\setminus \overline{\Sigma_{{\rm f}}}$.  
Let $\veu$ be a solution of the elastic wave equation \eqref{eq-elastic-special} with a priori norm \eqref{u-norm}. 
If for sufficiently small $\varepsilon_0$,
$$\|\veu\|_{H^2(U\times [-T,T])} \leq \varepsilon_0,$$
then for any $\alpha\in (0,\frac12)$, we have
$$\big\|\boldsymbol{\sigma}(\veu)|_{\Sigma_{{\rm f}}} \big\|_{H^{-\frac12-\alpha}(\Sigma_{{\rm f}}\times [-\frac{T}{2},\frac{T}{2}])}\leq C\Lambda_0^{1-\alpha}\varepsilon_1^{\alpha},$$
where $\varepsilon_1:=C(\log|\log\varepsilon_0|)^{-c}$. The same estimate also holds for the components $\sigma_n,\boldsymbol{\sigma}_{\tau}$. The constant $C$ depends on $\Lambda_0,T$, parameters of the elastic wave equation and $M, U, \Sigma_{{\rm f}}$, and $c$ is an absolute constant. 

\end{lemma}

\begin{proof}
First, we construct a connected open set $N\subset M$ with smooth boundary such that $U\subset N$ and $\Sigma_{{\rm f}}\subset \partial N$.
Let $\textbf{n}$ be a smooth unit normal vector field on $\Sigma_{{\rm f}}$, and consider the ray $z+r \textbf{n}$, $r\geq 0$, from a point $z\in \Sigma_{{\rm f}}$. Since $\overline{\Sigma_{{\rm f}}}$ is smoothly embedded in $\mathbb{R}^3$, there exists sufficiently small $r_0>0$ such that the nearest point of $z+r\textbf{n}$ in $\Sigma_{{\rm f}}$ is $z$ for all $r\in [0,r_0]$ and all $z\in \Sigma_{{\rm f}}$. 
This gives a smooth embedding $f: \Sigma_{{\rm f}}\times [0,r_0] \to \mathbb{R}^3$ such that $f(\Sigma_{{\rm f}}\times \{0\})=\Sigma_{{\rm f}}$. 
Note that $r_0$ is chosen such that the closure of the image of $f$ does not intersect with $\partial M$ or $\overline{U}$.
The boundary of the image $f(\Sigma_{{\rm f}}\times [0,r_0])$ extends $\Sigma_{{\rm f}}$, and can be smoothened over a neighborhood of $\partial \Sigma_{{\rm f}}$ without intersecting $\partial M$ or $\overline{U}$.
This construct a compact smooth manifolds $D\subset M$ with smooth boundary satisfying $\Sigma_{{\rm f}}\subset \partial D$, see Figure \ref{fig_N}. 
Note that the complement $M\setminus D$ is connected due to $\partial \Sigma_{{\rm f}}\neq \emptyset$.
Moreover, the manifold $D$ can be constructed such that the curvature tensor of $\partial D$ and its covariant derivatives are bounded depending on $r_0$, $d(\overline{\Sigma_{{\rm f}}},\partial M)$, $d(\overline{\Sigma_{{\rm f}}},\overline{U})$ and geometric parameters of $\Sigma_{{\rm f}}$.
Hence Proposition \ref{uc-M} is applicable to the connected open set $N:=M\setminus D$ with smooth boundary $\partial N=\partial M\cup \partial D$.

\smallskip
Applying Proposition \ref{uc-M} to $N=M\setminus D$ gives $\|\veu\|_{L^2(N \times [-\frac{T}{2},\frac{T}{2}])} \leq \varepsilon_1$. We have a priori norm $\|\veu\|_{H^1(N\times [-T,T])} \leq \Lambda_0$ by \eqref{u-norm}. By interpolation, we have
\begin{equation} \label{u-norm-small}
\|\veu\|_{H^{1-\alpha}(N\times [-\frac{T}{2},\frac{T}{2}])} \leq \Lambda_0^{1-\alpha} \varepsilon_1^{\alpha}, \quad\,\; \forall\,\alpha\in (0,1).
\end{equation}
Recall the system \eqref{eq-system} for $(\veu,v,\vew):=(\veu, \textrm{div}\, \veu, \textrm{curl}\, \veu)$.
Then 
\begin{equation} \label{regularity-vw}
\|v\|_{H^{-\alpha}(N\times [-\frac{T}{2},\frac{T}{2}])}+\|\vew\|_{H^{-\alpha}(N\times [-\frac{T}{2},\frac{T}{2}])} \leq 2\Lambda_0^{1-\alpha} \varepsilon_1^{\alpha}.
\end{equation}

To proceed further, we recall the $H_{(k,s)}$-norm in $\R^{n+1}$ (see \cite[Definition B.1.10]{HIII}) defined as
\begin{equation} \label{def-Hks}
\|u\|^2_{(k,s)} = \int_{\R^{n+1}} |\widehat u(\xi)|^2 (1+|\xi|^2)^k (1+|\xi'|^2)^s d\xi \, ,
\end{equation}
with respect to the coordinates $y=(x',x^n)\in \R^{n}\times \R$. 
Note that when $s=0$, the $H_{(k,0)}$-norm above is equivalent to the usual $H^k$-norm.
The idea is using partial hypoellipticity (e.g. \cite[Appendix B]{HIII} or \cite[Chapter 26.1]{E}) to trade regularity between normal and tangential components.

Since the $H_{(k,s)}$-norm was defined above for functions on $\mathbb{R}^{n+1}$, we apply the technique to local coordinates and patch up using partition of unity in the standard way.
Let $\{B_k\}_{k=1}^N$ be a finite open cover of a small neighborhood of $\Sigma_{{\rm f}}$ in $\overline{N}$, and let $\{\chi_k\}_{k=1}^N$ be a partition of unity subordinate to the open cover. Setting the diameter of each open set $B_k$ smaller than the injectivity radius of $\overline{N}$, one can work in the boundary normal coordinate $(y^1,y^2,y^3)\in \mathbb{R}^3_{y^1\geq 0}$ of $\overline{N}$, where $y^1=d(y,\partial N)$ is the coordinate normal to $\partial N$.
Let $\Phi_k:B_k \to \mathbb{R}^3_{y^1\geq 0}$ be the smooth (boundary normal) coordinate function on each $B_k$ such that $\Phi_k$ maps $\Sigma_{{\rm f}}\cap B_k$ to an open set of $\{(y^1,y^2,y^3)\in \mathbb{R}^3: y^1=0\}$.
Using the notation \eqref{def-Hks}, we write $v,\vew\in H_{(-\alpha,0)}$ from \eqref{regularity-vw} and thus $P_2 v=-A_2(\veu,v,\vew)\in H_{(-1-\alpha,0)}$, where $P_2=\frac{\rho}{2\mu+\lambda}\partial_t^2 -\Delta$ by the second equation in \eqref{eq-system}. 
We denote by $P_{2,k}$ the push-forward of the operator $P_2$ to $\Phi_k(B_k)$, that is, for any smooth function $\hat{f}$ on $\Phi_k(B_k)$,
\begin{equation}
P_{2,k} \hat{f}:= \big(P_2 (\hat{f} \circ \Phi_k) \big)\circ \Phi_k^{-1}.
\end{equation}
In the coordinate $\boldsymbol{y}=(y^1,y^2,y^3)$ of $\Phi_k(B_k)\subset \mathbb{R}^3_{y^1\geq 0}$, the operator $P_{2,k}$ has the standard form $P_{2,k}=\frac{\rho}{2\mu+\lambda}\partial_t^2 - \partial_{y^1}^2-a(\boldsymbol{y}, \partial_{y^2},\partial_{y^3})+\textrm{first order terms}$.
Then we apply \cite[Theorem B.2.9]{HIII} to $(\chi_k v)\circ \Phi_k^{-1}$ (which is a function in the half-plane supported in $\Phi_k(B_k)$),
\begin{eqnarray*}
\|(\chi_k v)\circ \Phi_k^{-1}\|_{(1-\alpha,-1)}&\leq& C\Big(\|(\chi_k v)\circ \Phi_k^{-1}\|_{(-\alpha,0)}+\|P_{2,k} \big((\chi_k v)\circ \Phi_k^{-1} \big)\|_{(-1-\alpha,0)} \Big) \\
&\leq& C\Big(\|(\chi_k v)\circ \Phi_k^{-1}\|_{(-\alpha,0)}+\|(\chi_k P_2 v) \circ \Phi_k^{-1}\|_{(-1-\alpha,0)} \\
&& + \| [P_2,\chi_k] v \circ \Phi_k^{-1}\|_{(-1-\alpha,0)}  \Big) \\
&\leq& C\Big(\|v\|_{(-\alpha,0)}+\|A_2(\veu,v,\vew)\|_{(-1-\alpha,0)} \Big) \\
&\leq& C\Lambda_0^{1-\alpha} \varepsilon_1^{\alpha},
\end{eqnarray*}
where the constants $C$ depend on the $	C^2$-norms of $\Phi_k, \Phi_k^{-1},\chi_k$, which, in turn, depend on the geometry of $\Sigma_{{\rm f}}$. Note that the commutator $[P_2,\chi_k]v$ in the above is first order in $v$.

Let $P_1=\frac{\rho}{\mu}\partial_t^2 -\Delta$ in the first equation $P_1\veu = -A_1(\veu,v)$ in \eqref{eq-system}, and denote by $P_{1,k}$ the push-forward of the operator $P_1$ to $\Phi_k(B_k)$.
Now using $(\chi_k\veu)\circ \Phi_k^{-1} \in H_{(1-\alpha,0)}\subset H_{(1-\alpha,-1)}$ and the regularity $(\chi_k v)\circ \Phi_k^{-1}\in H_{(1-\alpha,-1)}$ just obtained above, we have 
\begin{eqnarray*}
\|\big(\chi_k A_1(\veu,v)\big) \circ \Phi_k^{-1} \|_{(-\alpha,-1)} &\leq& C\Big( \|\nabla \big((\chi_k \veu)\circ \Phi_k^{-1}\big) \|_{(-\alpha,-1)} + \|\nabla \big((\chi_k v)\circ \Phi_k^{-1}\big) \|_{(-\alpha,-1)} \\
&& + \| (\chi_k \veu)\circ \Phi_k^{-1} \|_{(-\alpha,-1)} + \|(\chi_k v)\circ \Phi_k^{-1} \|_{(-\alpha,-1)} \Big) \\
&\leq & C\Big( \| (\chi_k \veu)\circ \Phi_k^{-1} \|_{(1-\alpha,-1)} + \|(\chi_k v)\circ \Phi_k^{-1} \|_{(1-\alpha,-1)} \Big) \\
&\leq & C\Big( \| \veu \|_{(1-\alpha,0)} + \|(\chi_k v)\circ \Phi_k^{-1} \|_{(1-\alpha,-1)} \Big) \\
&\leq& C\Lambda_0^{1-\alpha} \varepsilon_1^{\alpha},
\end{eqnarray*}
where the constants $C$ depend on the $	C^1$-norms of $\Phi_k, \Phi_k^{-1},\chi_k$. Note that we have used the fact that $A_1$ is a first-order linear operator in $\veu,v$.
Using \cite[Theorem B.2.9]{HIII} again to $(\chi_k \veu)\circ \Phi_k^{-1}$, we have
\begin{eqnarray*}
\|(\chi_k \veu)\circ\Phi_k^{-1}\|_{(2-\alpha,-1)}&\leq& C\Big(\|(\chi_k \veu)\circ\Phi_k^{-1}\|_{(1-\alpha,0)}+\|P_{1,k} \big( (\chi_k \veu)\circ \Phi_k^{-1} \big)\|_{(-\alpha,-1)} \Big) \\
&\leq& C\Big(\|\veu\|_{(1-\alpha,0)}+\|(\chi_k A_1(\veu,v))\circ \Phi_k^{-1}\|_{(-\alpha,-1)} +\| [P_1,\chi_k]\veu \circ \Phi_k^{-1}\|_{(-\alpha,-1)} \Big) \\
&\leq& C\Big(\|\veu\|_{(1-\alpha,0)}+\|\veu\|_{(1-\alpha,-1)}+\|(\chi_k A_1(\veu,v))\circ \Phi_k^{-1}\|_{(-\alpha,-1)} \Big) \\
&\leq& C \Lambda_0^{1-\alpha} \varepsilon_1^{\alpha}.
\end{eqnarray*}
This gives $\nabla \big((\chi_k \veu)\circ\Phi_k^{-1}\big) \in H_{(1-\alpha,-1)}$ and
$$\|\nabla \big((\chi_k \veu)\circ\Phi_k^{-1}\big)\|_{(1-\alpha,-1)}\leq C\Lambda_0^{1-\alpha}\varepsilon_1^{\alpha}.$$

Choosing $0<\alpha<1/2$, the trace of $\nabla \big((\chi_k \veu)\circ\Phi_k^{-1}\big)$ on $\Phi_k(B_k\cap \Sigma_{{\rm f}}) \times [-T/2,T/2]$ is well-defined in $H^{-\frac12-\alpha}$ by \cite[Theorem B.2.7]{HIII} with norm estimate,
$$\big\|\nabla \big((\chi_k \veu)\circ\Phi_k^{-1}\big)|_{\Phi_k(B_k\cap \Sigma_{{\rm f}})}\big\|_{H^{-\frac12-\alpha}}\leq C_{\alpha} \|\nabla \big((\chi_k \veu)\circ\Phi_k^{-1}\big)\|_{(1-\alpha,-1)} \leq C\Lambda_0^{1-\alpha}\varepsilon_1^{\alpha}.$$
Changing back to the original Euclidean coordinate, we have $\nabla (\chi_k \veu)|_{B_k\cap \Sigma_{{\rm f}}} \in H^{-\frac12-\alpha}$ on $(B_k\cap \Sigma_{{\rm f}})\times [-T/2,T/2]$. Since $\veu\in H^{1-\alpha} (N\times [-T/2,T/2])$, then $\veu |_{\Sigma_{{\rm f}}}\in L^2(\Sigma_{{\rm f}}\times [-T/2,T/2])$ if $\alpha<1/2$, and thus 
$$\chi_k \nabla \veu |_{B_k\cap \Sigma_{{\rm f}}} \in H^{-\frac12 -\alpha} \big((B_k\cap \Sigma_{{\rm f}})\times [-\frac{T}{2},\frac{T}{2}] \big),$$
with the same norm estimate as above.
Hence,
\begin{eqnarray*}
\|\nabla \veu |_{\Sigma_{{\rm f}}}\|_{H^{-\frac12 -\alpha} (\Sigma_{{\rm f}}\times [-\frac{T}{2},\frac{T}{2}])} &=& \Big\| \sum_{k=1}^N \chi_k \nabla \veu \big|_{\Sigma_{{\rm f}}} \Big\|_{H^{-\frac12 -\alpha} (\Sigma_{{\rm f}}\times [-\frac{T}{2},\frac{T}{2}])} \\
&\leq& N \max_k \big \|\chi_k \nabla \veu |_{B_k\cap \Sigma_{{\rm f}}} \big\|_{H^{-\frac12 -\alpha} (\Sigma_{{\rm f}}\times [-\frac{T}{2},\frac{T}{2}])} \\
&\leq& CN\Lambda_0^{1-\alpha}\varepsilon_1^{\alpha}.
\end{eqnarray*}
Then the regularity for the stress tensor $\boldsymbol{\sigma}(\veu)$ on $\Sigma_{{\rm f}}$ immediately follows from the definition \eqref{def-stress}.
Contracting the stress tensor $\boldsymbol{\sigma}(\veu)$ gives the same estimate for the components $\sigma_n,\boldsymbol{\sigma}_{\tau}$.
\end{proof}

To prove our main result we need the following lemma.

\begin{lemma} \label{lemma-H1}
Let $c >0$ be a constant and let $\Omega \subset \R^n$, $n > 2$, be a bounded open set. 
Suppose $\boldsymbol{f} \in H^1(\Omega; \mathbb{R}^m)$ with value in $\mathbb{R}^m$, $m\geq 1$.
Let $\boldsymbol{\rho}: \R^m \to \R^m$ be defined by
    \begin{align} \label{def-rho-vector}
\boldsymbol{\rho}(\xi) = \frac{\xi}{\max(|\xi|,c)}, \quad \xi=(\xi_1,\cdots,\xi_m)\in \R^m.
    \end{align}
Denote by $T_{\boldsymbol{\rho}} \boldsymbol{f}=\boldsymbol{\rho} \circ \boldsymbol{f}$ the composition operator.    
Then $T_{\boldsymbol{\rho}}$ is a continuous operator from $H^1(\Omega;\R^m)$ to $H^1(\Omega;\R^m)$, and
$$\|T_{\boldsymbol{\rho}} \boldsymbol{f}\|_{H^1(\Omega;\R^m)}\leq C(m,c,\Omega)\Big(1+ \|\boldsymbol{f}\|_{H^1(\Omega;\R^m)} \Big).$$
\end{lemma}
\begin{proof}
Let $j = 1, \cdots, m$.
To show $T_{\boldsymbol{\rho}} \boldsymbol{f}\in H^1(\Omega; \mathbb{R}^m)$, it is enough to show that $\rho_j \circ \boldsymbol{f} \in H^1(\Omega)$ for all $j$, where the function $\rho_j: \R^m \to \R$ is the component of $\boldsymbol{\rho}$, namely,
    \begin{align} \label{def-rho}
\rho_j(\xi) = \frac{\xi_j}{\max(|\xi|,c)}, \quad \xi=(\xi_1,\cdots,\xi_m)\in \R^m.
    \end{align}
According to \cite[Theorem 1]{MM}, this follows after showing that $\rho_j$ is locally Lipschitz and that there is a uniform constant $C > 0$ such that 
    \begin{align}
|\p_{\xi_k} \rho_j(\xi)| \leq C, \quad k = 1,\cdots,m,
    \end{align}
almost everywhere in $\R^m$, for any $j=1,\cdots,m$.

For $|\xi|\neq c$, the function $\rho_j$ is smooth and it is straightforward to check that 
$$|\p_{\xi_k} \rho_j | \leq \frac 2 c, \quad j,k=1,\cdots,m.$$
Hence it suffices to verify that $\rho_j$ is Lipschitz across $|\xi| = c$.
Let $|\xi| \leq c$ and $|\eta| \geq c$. Then
    \begin{align*}
\rho_j(\xi) - \rho_j(\eta) = \frac{\xi_j}c - \frac{\eta_j}{|\eta|},
    \end{align*}
and
    \begin{align*}
|\rho_j(\xi) - \rho_j(\eta)| 
&\leq \left| \frac{\xi_j}{c}-\frac{\xi_j}{|\eta|}\right| + \left| \frac{\xi_j}{|\eta|}-\frac{\eta_j}{|\eta|} \right| 
\leq \frac1c \big| |\eta|-c| \big| + \frac1c |\xi_j-\eta_j| \\
&\leq \frac1c \big| |\eta|-|\xi| \big| + \frac1c |\xi_j-\eta_j|
\leq \frac2c |\xi-\eta|.
    \end{align*}
This shows that $\rho_j$ is Lipschitz on $\mathbb{R}^m$ with uniform Lipschitz constant $2/c$, for all $j=1,\cdots,m$.    
Moreover, \cite[Theorem 1]{MM} gives an estimate for the norm $\|\rho_j \circ \boldsymbol{f}\|_{H^1(\Omega)}$,
$$\|\rho_j\circ \boldsymbol{f}\|_{H^1(\Omega)} \leq C(m,c,\Omega)\Big(1+ \|\boldsymbol{f}\|_{H^1(\Omega;\R^m)} \Big),\quad j=1,\cdots,m.$$

For the continuity of $T_{\boldsymbol{\rho}}$, it is necessary to use the specific form of the function $\boldsymbol{\rho}$ in \eqref{def-rho-vector}, as the continuity does not hold for all composition operators in general, see e.g. \cite[Section 1]{M}.
In our case, for $\xi=(\xi_1,\cdots,\xi_m)\in \R^m$, we define
\begin{equation}
\boldsymbol{h}(\xi):=\boldsymbol{\rho}(\xi)-\frac{\xi}{c} = \left\{ \begin{aligned} & \frac{\xi}{|\xi|}-\frac{\xi}{c} ,\quad &|\xi|\geq c, \\
&0,\quad &|\xi|<c.
\end{aligned} \right.
\end{equation}
It is clear that $\boldsymbol{h}$ is uniformly Lipschitz on $\mathbb{R}^m$ as $\boldsymbol{\rho}$ is uniformly Lipschitz on $\R^m$.
We apply \cite[Example 3.1]{M} to the Lipschitz function $\boldsymbol{h}$ with $S=\{\xi\in \R^m: |\xi|\leq c\}$, the closed ball of radius $c$. Since $\boldsymbol{h}$ is identically $0$ in $S$ and $\boldsymbol{h}$ is smooth in the complement of $S$, it is enough to verify that, for every $\xi_0\in \partial S$, 
\begin{equation} \label{verify-condition}
\lim_{\substack{\xi \to \xi_0 \\ \xi\notin \partial S}} \nabla \boldsymbol{h}(\xi) \cdot \tau=0,\quad \textrm{ for all }\tau\in T_{\partial S}(\xi_0),
\end{equation}
where $T_{\partial S}(\xi_0)$ denotes the tangent space of $\partial S=\{\xi\in \R^m: |\xi|=c\}$ at $\xi_0\in \partial S$. For $\xi\in S\setminus \partial S$, we see $\nabla \boldsymbol{h}(\xi)=0$ since $\boldsymbol{h}$ is identically $0$ in $S$. On the complement $\R^m\setminus S$, since $\boldsymbol{h}|_{\R^m \setminus S}$ can be smoothly extended to a function $\widetilde{\boldsymbol{h}}(\xi):=\xi/|\xi|-\xi/c$ on a neighborhood of $\partial S$, we have
$$\lim_{\substack{\xi\to \xi_0 \\ \xi\notin S}} \nabla \boldsymbol{h}(\xi) \cdot \tau=\nabla \widetilde{\boldsymbol{h}}(\xi_0)\cdot \tau=0,\quad \textrm{ for all }\tau\in T_{\partial S}(\xi_0).$$
In the above, the fact that $\nabla \widetilde{\boldsymbol{h}}(\xi_0)\cdot \tau=0$ is due to the definition of tangent vector, namely, by differentiating $(\widetilde{h}_j \circ \gamma)(s)$ at $s=0$, where $\widetilde{h}_j$, $j=1,\cdots,m$, are the components of $\widetilde{\boldsymbol{h}}$ and $\gamma(s)$ is a smooth curve in $\partial S$ with the initial vector $\gamma'(0)=\tau\in T_{\partial S}(\xi_0)$. Then the claim immediately follows since $\widetilde{h}_j \circ \gamma=0$ due to $\widetilde{\boldsymbol{h}}=0$ on $\partial S$.
This verifies \eqref{verify-condition}, and then \cite[Example 3.1]{M} gives the continuity of $T_{\boldsymbol{h}}$ and thus the continuity of $T_{\boldsymbol{\rho}}$.
\end{proof}

Now we prove Theorem \ref{main}.

\begin{proof}[Proof of Theorem \ref{main}]
Since $\big| \jmp{\p_t \veu_{\tau}}\big| \geq c_1>0$ on $\Sigma_{{\rm f}}\times [-T,T]$, i.e., the fault is slipping everywhere, it follows from \eqref{apriori-fault} that
\begin{equation} \label{eq-friction-Lip}
g=|\boldsymbol{\sigma}_{\tau}| =\mathscr{F}|F_n| \in C^{0,1}(\Sigma_{{\rm f}}\times [-T,T]), \quad \; \|g\|_{C^{0,1}(\Sigma_{{\rm f}}\times [-T,T])}\leq C_0^2.
\end{equation}
On the other hand, the friction condition \eqref{Tresca} implies that
\begin{equation} \label{friction-direction}
\boldsymbol{\sigma}_{\tau}=g \, \frac{\jmp{\p_t \veu_{\tau}}}{\big| \jmp{\p_t \veu_{\tau}} \big|}.
\end{equation}
Denote
\begin{equation}
\veu_+:= \veu |_{N},\quad \veu_-:=\veu|_{\textrm{int}(D)}, 
\end{equation}
where $N,D$ are the manifolds with smooth boundary constructed in the proof of Lemma \ref{lemma-trace}, see Figure \ref{fig_N}.
Let us denote
\begin{equation}
E(N):=W^{1,\infty}(-T,T;H^1(N))\cap W^{2,\infty}(-T,T;L^2(N)).
\end{equation}
As $\veu_+ \in E(N)$ by \eqref{regularity-direct}, we can extend $\veu_+$ to a small neighborhood $\widetilde{N}$ of $N$ in $M$ so that the extension of $\veu_+$ is in the same regularity class $E(\widetilde{N})$.
Similarly, we can extend $\veu_-$ to a small neighborhood $\widetilde{D}$ of $D$ in $M$ so that the extension of $\veu_-$ is in the same regularity class $E(\widetilde{D})$.
Let $V$ be a small open neighborhood of $\Sigma_{{\rm f}}$ in $M$.
As $\Sigma_{{\rm f}}\subset \partial D$, in this way, the functions $\veu_+,\veu_-$ are both extended to $V$ and the extensions are in $E(V)$.
Then $\partial_t \veu_+,\partial_t \veu_- \in H^1(V\times [-T,T])$, and thus the tangential component of the difference
\begin{equation}
\boldsymbol{f}:=\partial_t(\veu_+ -\veu_-)_{\tau} \in H^1(V\times [-T,T]).
\end{equation}
Recall the definition of tangential component near $\Sigma_{{\rm f}}$ in \eqref{tangential-jump} where the unit normal vector field $\boldsymbol{n}$ is smoothly extended to a neighborhood of $\Sigma_{{\rm f}}$.
By Lemma \ref{lemma-H1}, the function $\boldsymbol{f}/\max(| \boldsymbol{f}|,c_1) \in H^1(V\times [-T,T])$. Hence, as $V$ is an open neighborhood of $\Sigma_{{\rm f}}$ in $M$, the trace of $\boldsymbol{f}/\max(|\boldsymbol{f}|,c_1)$ onto $\Sigma_{{\rm f}}$ is in $H^{1/2}(\Sigma_{{\rm f}}\times [-T,T])$. 

\smallskip
Denote the composition operator $T_{\boldsymbol{\rho}} \boldsymbol{f}:=\boldsymbol{\rho}\circ \boldsymbol{f}$, where $\boldsymbol{\rho}:\mathbb{R}^3\to \R^3$ is the Lipschitz function with uniform Lipschitz constant $4/c_1$ defined in \eqref{def-rho-vector}, taking $c=c_1$. In the above we have shown that $\textrm{tr}(T_{\boldsymbol{\rho}} \boldsymbol{f}) \in H^{1/2} (\Sigma_{{\rm f}}\times [-T,T])$, where $\textrm{tr}$ stands for the trace operator onto $\Sigma_{{\rm f}}$. Now we show that $\textrm{tr}(T_{\boldsymbol{\rho}} \boldsymbol{f})=T_{\boldsymbol{\rho}} (\textrm{tr}\boldsymbol{f})$ in $L^2(\Sigma_{{\rm f}}\times [-T,T])$, which would imply that $T_{\boldsymbol{\rho}} (\textrm{tr}\boldsymbol{f}) \in H^{1/2} (\Sigma_{{\rm f}}\times [-T,T])$.
Let $\boldsymbol{f}_k\in C^{\infty}(\overline{V}\times [-T,T])$ be a sequence of smooth vector-valued functions such that $\boldsymbol{f}_k\to \boldsymbol{f}$ as $k\to \infty$ in $H^1(V\times [-T,T])$.
Then as $k\to \infty$,
\begin{eqnarray*}
\| T_{\verho} (\textrm{tr}\boldsymbol{f})-T_{\verho}(\textrm{tr} \boldsymbol{f}_k) \|_{L^2(\Sigma_{{\rm f}}\times [-T,T])} &\leq& \frac{4}{c_1} \|\textrm{tr}\boldsymbol{f}-\textrm{tr} \boldsymbol{f}_k\|_{L^2(\Sigma_{{\rm f}}\times [-T,T])} \\
&\leq& \frac{4}{c_1} C\|\boldsymbol{f}-\boldsymbol{f}_k\|_{H^1(V\times [-T,T])} \to 0,
\end{eqnarray*}
and by the continuity of the composition operator $T_{\verho}$ from Lemma \ref{lemma-H1},
\begin{equation*}
\| \textrm{tr}(T_{\verho}\boldsymbol{f}) -\textrm{tr} (T_{\verho}\boldsymbol{f}_k) \|_{L^2(\Sigma_{{\rm f}}\times [-T,T])} \leq C \| T_{\verho}\boldsymbol{f} -T_{\verho}\boldsymbol{f}_k \|_{H^1(V\times [-T,T])} \to 0.
\end{equation*}
Since $T_{\verho}(\textrm{tr} \boldsymbol{f}_k)=\textrm{tr} (T_{\verho}\boldsymbol{f}_k)$ for smooth functions, we see that $T_{\verho}(\textrm{tr} \boldsymbol{f})=\textrm{tr} (T_{\verho}\boldsymbol{f})$ in $L^2(\Sigma_{{\rm f}}\times [-T,T])$.
As $\textrm{tr}\boldsymbol{f}=\jmp{\partial_t \veu_{\tau}}$ by definition, we conclude that $T_{\verho}(\jmp{\partial_t \veu_{\tau}})\in H^{1/2} (\Sigma_{{\rm f}}\times [-T,T])$.
Thus, due to the condition $\big| \jmp{\p_t \veu_{\tau}}\big| \geq c_1>0$ on $\Sigma_{{\rm f}}\times [-T,T]$, we have
\begin{equation} \label{H1-direction}
\frac{\jmp{\p_t \veu_{\tau}}}{\big| \jmp{\p_t \veu_{\tau}} \big|} \in H^{\frac12}(\Sigma_{{\rm f}}\times [-T,T]).
\end{equation}

Using the equation \eqref{friction-direction}, \eqref{eq-friction-Lip} and \eqref{H1-direction} imply that $\boldsymbol{\sigma}_{\tau}\in H^{1/2}(\Sigma_{{\rm f}}\times [-T,T])$.
Namely, denote by $A_g$ the multiplication operation by $g$.
Since $g$ is Lipschitz by \eqref{eq-friction-Lip}, the operator
$A_g:L^2(\Sigma_{{\rm f}}\times [-T,T])\to L^2(\Sigma_{{\rm f}}\times [-T,T])$ and $A_g:H^1(\Sigma_{{\rm f}}\times [-T,T])\to H^1(\Sigma_{{\rm f}}\times [-T,T])$ are bounded. 
Hence
$A_g:H^s(\Sigma_{{\rm f}}\times [-T,T])\to H^s(\Sigma_{{\rm f}}\times [-T,T])$
is a bounded operator for any $s\in (0,1)$ by \cite[Theorem 5.1]{LM}.
Thus, by \eqref{friction-direction} and \eqref{H1-direction},
\begin{equation} \label{stress-H1/2}
\boldsymbol{\sigma}_{\tau}\in H^{\frac12}(\Sigma_{{\rm f}}\times [-T,T]),
\end{equation}
with an estimate of the norm depending on $\Lambda_0,C_0,c_1$ and $\Sigma_{{\rm f}}$.

\smallskip
Suppose that we have two systems with friction coefficients $\mathscr{F}_1, \mathscr{F}_2$ at $\Sigma_{{\rm f}}$,  and we are given elastic waves $\veu_1,\veu_2$ that are close in the sense that
$$\|\veu_1-\veu_2\|_{H^2(U\times [-T,T])}\leq \varepsilon_0.$$
Denote 
by $\sigma_n^{(j)}, \boldsymbol{\sigma}_{\tau}^{(j)}$ $(j=1,2)$ the components of the stress tensors corresponding to the two systems.
Since the stress tensor $\boldsymbol{\sigma}(\veu)$ is linear in $\veu$, applying Lemma \ref{lemma-trace} to $\veu_1-\veu_2$ gives
$$\Big\| \boldsymbol{\sigma}_{\tau}^{(1)}-\boldsymbol{\sigma}_{\tau}^{(2)}\Big\|_{H^{-\frac12-\alpha}(\Sigma_{{\rm f}}\times [-\frac{T}{2},\frac{T}{2}])} \leq C\Lambda_0^{1-\alpha}\varepsilon_1^{\alpha},$$ 
$$\Big\| \sigma_{n}^{(1)}-\sigma_{n}^{(2)}\Big\|_{H^{-\frac12-\alpha}(\Sigma_{{\rm f}}\times [-\frac{T}{2},\frac{T}{2}])} \leq C\Lambda_0^{1-\alpha}\varepsilon_1^{\alpha}.$$
Then picking $\alpha=1/4$, by interpolation with $\boldsymbol{\sigma}_{\tau}\in H^{1/2}$ by \eqref{stress-H1/2} and $\sigma_n=F_n \in C^{0,1}\subset H^1$ by \eqref{apriori-fault}, we have
\begin{eqnarray*}
\Big\| \boldsymbol{\sigma}_{\tau}^{(1)}-\boldsymbol{\sigma}_{\tau}^{(2)} \Big\|_{L^2(\Sigma_{{\rm f}}\times [-\frac{T}{2},\frac{T}{2}])} &\leq& C(C_0,\Lambda_0,T,c_1)\varepsilon_1^{\frac{1}{10}}, \\ \Big\| \sigma_{n}^{(1)}-\sigma_{n}^{(2)} \Big\|_{L^2(\Sigma_{{\rm f}}\times [-\frac{T}{2},\frac{T}{2}])} &\leq& C(C_0,\Lambda_0,T)\varepsilon_1^{\frac{1}{7}}.
\end{eqnarray*}
Hence,
\begin{eqnarray*}
\|\mathscr{F}_1-\mathscr{F}_2\|_{L^2(\Sigma_{{\rm f}}\times [-\frac{T}{2},\frac{T}{2}])} &=& \Big\| \frac{|\boldsymbol{\sigma}_{\tau}^{(1)}|}{|\sigma_{n}^{(1)}|} -\frac{|\boldsymbol{\sigma}_{\tau}^{(2)}|}{|\sigma_{n}^{(2)}|} \Big\|_{L^2(\Sigma_{{\rm f}}\times [-\frac{T}{2},\frac{T}{2}])} \\
&\leq& c_0^{-2} \Big\|   \Big(|\boldsymbol{\sigma}_{\tau}^{(1)}|-|\boldsymbol{\sigma}_{\tau}^{(2)}| \Big) |\sigma_n^{(2)}| - \Big(|\sigma_{n}^{(1)}|-|\sigma_{n}^{(2)}| \Big) |\boldsymbol{\sigma}_{\tau}^{(2)}|  \Big\|_{L^2} \\
&\leq& c_0^{-2}C_0 \Big\| |\boldsymbol{\sigma}_{\tau}^{(1)}|-|\boldsymbol{\sigma}_{\tau}^{(2)}| \Big\|_{L^2} + c_0^{-2}C_0^2 \Big\| |\sigma_{n}^{(1)}|-|\sigma_{n}^{(2)}| \Big\|_{L^2} \\
&\leq& c_0^{-2}C_0 \Big\| \boldsymbol{\sigma}_{\tau}^{(1)}-\boldsymbol{\sigma}_{\tau}^{(2)}\Big\|_{L^2(\Sigma_{{\rm f}}\times [-\frac{T}{2},\frac{T}{2}])} + c_0^{-2}C_0^2 \Big\| \sigma_{n}^{(1)}-\sigma_{n}^{(2)} \Big\|_{L^2(\Sigma_{{\rm f}}\times [-\frac{T}{2},\frac{T}{2}])} \\
&\leq& C(c_0,C_0,\Lambda_0,T,c_1) \varepsilon_1^{\frac{1}{10}},
\end{eqnarray*}
which proves the stability part (2).
The uniqueness part (1) is a consequence of the stability when $\varepsilon_0 \to 0$.
\end{proof}


\bigskip
\begin{thebibliography}{99}

\bibitem{AF} R. Adams, J. Fournier, \emph{Sobolev spaces}, 2nd edition, Academic Press, 2003.


\bibitem{B20}
E. Brodsky et al., \emph{The state of stress on the fault before, during, and after a major earthquake}, Annual Review of Earth and Planetary Sciences \textbf{48} (2020), 49--74.

\bibitem{BILL}
D. Burago, S. Ivanov, M. Lassas, J. Lu, \emph{Quantitative stability of Gel’fand’s inverse boundary problem}, to appear in Anal. PDE., arXiv:2012.04435v3.

\bibitem{CDM}
M. Causse, L.A. Dalguer, P. M. Mai, \emph{Variability of dynamic source parameters inferred from kinematic models of past earthquakes}, Geophysical Journal International \textbf{196} (2013), 1754--1769.

\bibitem{DLLO}
M. de Hoop, M. Lassas, J. Lu, L. Oksanen, \emph{Quantitative unique continuation for the elasticity system with application to the kinematic inverse rupture problem}, Comm. PDE. \textbf{48} (2023), 286--314.

\bibitem{DL}
G. Duvaut, J. L. Lions, \emph{Inequalities in mechanics and physics}, Springer, 1976.


\bibitem{EINT} 
M. Eller, V. Isakov, G. Nakamura, D. Tataru, \emph{Uniqueness and stability in the Cauchy problem for Maxwell and elasticity systems}, Studies in Mathematics and its Applications \textbf{31} (2002), 329-349.

\bibitem{E}
G. Eskin, \emph{Lectures on linear partial differential equations}, AMS, 2011.

\bibitem{Evans}
L. Evans, \emph{Partial differential equations}, AMS, 1998.

\bibitem{FSG}
W. Fan, P. Shearer, P. Gerstoft, \emph{Kinematic earthquake rupture inversion in the frequency domain}, Geophys. J. Int. \textbf{199} (2014), 1138--1160.

\bibitem{HI}
S. Hirano, H. Itou, \emph{Parameter interdependence of dynamic self-similar crack with distance-weakening friction}, Geophys. J. Int. \textbf{223} (2020), 1584--1596.

\bibitem{HIII}
L. H\"ormander, \emph{The analysis of linear partial differential operators III}, Springer, 1985.

\bibitem{IT}
S. Ide, M. Takeo, \emph{Determination of constitutive relations of fault slip based on seismic waves analysis}, J. geophys. Res. \textbf{102} (1997), 27379--27391.

\bibitem{IK}
H. Itou, T. Kashiwabara, \emph{Unique solvability of crack problem with time-dependent friction condition in linearized elastodynamic body}, Mathematical notes of NEFU \textbf{28} (2021), 121--134.

\bibitem{KI22}
T. Kashiwabara, H. Itou, \emph{Unique solvability of a crack problem with Signorini-type and Tresca friction conditions in a linearized elastodynamic body}, Phil. Trans. R. Soc. A \textbf{380}:20220225, 2022.

\bibitem{LY}
P. Li, S. T. Yau, \emph{Estimates of eigenvalues of a compact Riemannian manifold}, Proceedings of Symposia in Pure Math. \textbf{36} (1980), 205--239.

\bibitem{LM}
J. L. Lions, E. Magenes, \emph{Non-homogeneous boundary value problems and applications Vol. I}, Die Grundlehren der mathematischen Wissenschaften \textbf{181}, Springer-Verlag, Berlin Heidelberg, 1972.

\bibitem{MM}
M. Marcus and V. J. Mizel. \emph{Complete characterization of functions which act, via superposition, on Sobolev spaces}, Trans. Amer. Math. Soc. \textbf{251} (1979), 187--218.


\bibitem{M} R. Musina, \emph{On the continuity of the Nemitsky operator induced by a Lipschitz continuous map}, Proc. Amer. Math. Soc. \textbf{111} (1991), 1029--1041.

\bibitem{SM}
M. Sofonea, A. Matei, \emph{Mathematical models in contact mechanics}, London Mathematical Society Lecture Note Series \textbf{398}, Cambridge University Press, 2012.

\bibitem{Tataru1}
D.\ Tataru, \emph{Unique continuation for solutions to PDE's; between H\"ormander's theorem and Holmgren's theorem}, Comm. PDE. \textbf{20}  (1995), 855--884.

\bibitem{YZ}
H. Yue, Y. Zhang, Z. Ge, T. Wang, L. Zhao, \emph{Resolving rupture processes of great earthquakes: reviews and perspective from fast response to joint inversion}, Sci. China Earth Sci. \textbf{63} (2020), 492--511. 

\bibitem{ZY}
Z. Zhao, H. Yue, \emph{A two-step inversion for fault frictional properties using a temporally varying afterslip model and its application to the 2019 Ridgecrest earthquake}, Earth and Planetary Science Letters \textbf{602} (2023), 117932.

\end{thebibliography}
\end{document}